\documentclass[EJP, preprint]{ejpecp} 

\usepackage{amsfonts, amssymb, amsmath, mathtools, dsfont}
\usepackage{enumerate}

\usepackage{color}
\usepackage{graphicx}

\usepackage{subcaption} 
\usepackage[font=small]{caption} 
\SHORTTITLE{Random matrices with independent columns}

\TITLE{Random matrices acting on sets: Independent columns\support{R.V.~is partially supported by NSF Grant DMS 2451011 and NSF+Simons Research Collaborations on the Mathematical and Scientific Foundations of Deep Learning; Y.P.~is partially supported by an NSERC Discovery Grant (GR009284), an NSERC Discovery Accelerator Supplement (GR007657), and a Tier II Canada Research Chair in Data Science (GR009243). }\
    
      } 

\AUTHORS{%
  Yaniv~Plan\footnote{University of British Columbia, Canada.
    \EMAIL{yaniv@math.ubc.ca}}
  \and 
  Roman~Vershynin\footnote{University of California, Irvine, USA. \EMAIL{rvershyn@uci.edu} }}


\AMSSUBJ{15A42; 60D05} 

\SUBMITTED{February 23, 2025} 

\VOLUME{0}
\YEAR{2023}
\PAPERNUM{0}
\DOI{10.1214/YY-TN}

\ABSTRACT{We study random matrices with independent subgaussian columns. Assuming each column has a fixed Euclidean norm, we establish conditions under which such matrices act as near-isometries when restricted to a given subset of their domain. We show that, with high probability, the maximum distortion caused by such a matrix is proportional to the Gaussian complexity of the subset, scaled by the subgaussian norm of the matrix columns. This linear dependence on the subgaussian norm is a new phenomenon, as random matrices with independent rows or independent entries typically exhibit superlinear dependence. As a consequence, normalizing the columns of random sparse matrices leads to stronger embedding guarantees.}

\DeclareMathOperator{\E}{\mathbb{E}}

\DeclareMathOperator{\Ber}{Ber}
\DeclareMathOperator{\Binom}{Binom}

\renewcommand{\Pr}[2][]{\mathbb{P}_{#1} \left\{ #2 \rule{0mm}{3mm}\right\}}

\DeclarePairedDelimiter \abs{\lvert}{\rvert} 
\DeclarePairedDelimiter \norm{\lVert}{\rVert}
\DeclarePairedDelimiterX \ip[2]{\langle}{\rangle}{#1,#2}
\DeclarePairedDelimiterXPP \Prob[1]{\mathbb{P}}\{\}{}{
   
   #1} 
\DeclarePairedDelimiterXPP \Probevent[1]{\mathbb{P}}(){}{
   
   #1} 

\def \R {\mathbb{R}}

\def \d {\delta}
\def \l {\lambda}

\def \tran {\mathsf{T}}

\def \psitwo {{\psi_2}}

\def \rad {\mathrm{rad}}

\theoremstyle{remark}

\begin{document}



\section{Introduction}  \label{sec: intro}

Let $A \in \R^{m\times n}$ be a random matrix and $T \subset \R^n$ a fixed set.  We study the question: \textit{Under what circumstances is $A$ well conditioned when restricted to $T$?}  More technically, we study the quantity
\[
\delta \coloneqq \sup_{x\in T} \abs[\Big]{\norm*{Ax}_2 - \norm*{x}_2}.
\]
Observe that the \textit{isometry constant} $\delta$ is the smallest constant such that
$$
 \norm*{x}_2 - \delta \leq \norm*{Ax}_2 \leq \norm*{x}_2 + \delta \qquad \mbox{for all } x \in T.
$$
If $\delta$ is small, $A$ embeds $T$ into $\R^m$ without significantly distorting Euclidean norms. If $T$ is replaced by the Minkowski difference $T-T$, then the embedding approximately preserves Euclidean distances between points in $T$, so we can  
think of $A$ as a \textit{near isometry} on $T$. Random matrices that are near isometries are useful in a variety of applications, including compressed sensing \cite{CS-book, eldar2012compressed}, dimension reduction as in the Johnson-Lindenstrauss lemma \cite{johnson1984extensions}, and randomized sketching \cite{yang2017randomized, 2015pilanci}.  The subset $T$ of interest varies by application.   When $T$ belongs to the sphere and $\delta < 1$, the null space of $A$ is a random subspace that must not intersect $T$, and thus the study of $\delta$ also has interesting implications regarding the behaviour of random subspaces \cite{jeong2022sub, LMPV}.  

If $A$ is a random Gaussian matrix, the size of $\delta$ is well understood \cite{Schechtman2006}.  It is controlled by the Gaussian width or Gaussian complexity of $T$ \cite{hdp-book}.

\begin{definition}[The Gaussian width and Gaussian complexity of a set]
    The \textit{Gaussian width} $w(T)$ and the \textit{Gaussian complexity} $\gamma(T)$ of a set $T \subset \R^n$ are defined as follows:
    $$
    w(T) \coloneqq \E \sup_{x \in T} \ip{g}{x}, \quad
    \gamma(T) \coloneqq \E \sup_{x \in T} \abs{\ip{g}{x}} 
    $$
    where $g$ is a standard normal random vector in $\R^n$, i.e. $g \sim N(0, I_n)$.
\end{definition}

These two quantities are identical for symmetric sets, and in general we have $w(T) + \rad(T) \asymp \gamma(T)$, where $\rad(T):= \sup_{x\in T} \norm{x}_2$. For more properties of Gaussian width and complexity, see \cite{hdp-book}.

To extend our discussion to more general random matrix ensembles, recall that the \textit{subgaussian norm} of a random variable $Z$ is defined as 
$$
\norm*{Z}_{\psi_2} := \inf\left\{t > 0: \E \exp(Z^2/t^2) \leq 2\right\}
$$
and $Z$ is called \textit{subgaussian} if $\norm*{Z}_{\psi_2} < \infty$. Further, the subgaussian norm of a random vector $V \in \R^n$ is defined as 
\begin{equation} \label{eq: sub-gauss vector}
    \norm{V}_{\psi_2}:=\sup_{x \in S^{n-1}} \norm{\ip{V}{x}}_{\psi_2}
\end{equation}
i.e.~it is the largest subgaussian norm of one-dimensional marginals of $V$. For basic introduction to subgaussian distributions, see \cite{hdp-book}.

Matrices with subgaussian entries (or rows or columns) represent a notable generalization beyond the Gaussian case. In particular, this generalization includes discrete and/or sparse random matrices, which can be used to accelerate computations \cite{kane2014sparser} and are essential in applications like group testing \cite{cohen2021multi} and wireless network activity detection \cite{kueng2017robust}.

When $A$ has independent, isotropic,\footnote{Recall that a random vector $V\in \R^n$ is called \textit{isotropic} if $\E V V^\tran = I_n$ where $I_n$ is the identity matrix.} and subgaussian rows, the quantity $\delta$ has similar guarantees to the Gaussian case as shown in the following theorem. 

\begin{theorem}[Corollary 1.2 in \cite{jeong2022sub}] \label{thm: old result}
    Let $A$ be an $m \times n$ matrix whose rows are independent, isotropic, subgaussian random vectors in $\R^n$ with subgaussian norms bounded by $K$. Then, for any subset $T \subset \R^n$, we have
    $$
    \E \sup_{x \in T} \abs[\Big]{\norm*{Ax}_2 - \sqrt{m}\norm*{x}_2} 
    \le C K \sqrt{\log K} \, \gamma(T).
    $$
    Moreover, for any $u \geq 0$, with probability at least $1 - 3e^{-u^2}$, we have
    $$
    \sup_{x \in T} \abs[\Big]{\norm*{Ax}_2 - \sqrt{m}\norm*{x}_2} 
    \le C K \sqrt{\log K} \left[w(T) + u\cdot \rad(T)\right].
    $$
    Here and throughout the paper, $C>0$ denotes an absolute constant.
\end{theorem}

\medskip

This theorem builds directly on the work in \cite{LMPV}, which presents the same result but with $K \sqrt{\log K}$ replaced by $K^2$. In fact, the $K \sqrt{\log K}$ dependence is known to be optimal in this setting \cite{jeong2022sub}. Similar preceding results also appear in the papers \cite{2005Klartag, 2007Mendelson, 2015dirksen}; see \cite[Section 3]{LMPV} for a discussion of their relationship.

In this paper, we develop an analogous theory for matrices with independent subgaussian columns.  While there are several studies of random matrices with independent subgaussian columns restricted to specific subsets---such as the set of sparse vectors \cite{adamczak2011restricted}, finite sets \cite{cohen2018simple}, or the entire sphere \cite{V-rmt}---we are only aware of one result that allows arbitrary sets $T$, due to  D.~Bartl and S.~Mendelson \cite{bartl2022random}.  We now state our main theorem and corollary, and then compare to the results in \cite{bartl2022random} in Remark \ref{rem: comparison} below.



\begin{theorem}			\label{thm: main}
    Let $A$ be an $m \times n$ matrix whose columns $A_i$ are independent, mean zero, subgaussian random vectors in $\R^m$ satisfying $\norm*{A_i}_2 = 1$ almost surely and $\norm*{A_i}_{\psi_2} \leq K$. Then, for any subset $T \subset \R^n$, we have
    $$
    \E \sup_{x \in T} \abs[\Big]{\norm*{Ax}_2 - \norm*{x}_2} \le C K \gamma(T).
    $$
    Moreover, for any $u \geq 0$, with probability at least $1 - 3e^{-u^2}$, we have
    $$
    \sup_{x \in T} \abs[\Big]{\norm*{Ax}_2 - \norm*{x}_2} \le C K \left[w(T) + u\cdot \rad(T)\right].
    $$ 
\end{theorem}

\medskip
A trivial rescaling argument yields the following:

\begin{corollary} \label{cor: lambda cols}
    Let $A$ be an $m \times n$ matrix whose columns $A_i$ are independent, mean zero, subgaussian random vectors in $\R^n$ satisfying $\norm*{A_i}_2 = \lambda$ almost surely and $\norm*{A_i}_{\psi_2} \leq K$. Then, for any subset $T \subset \R^n$, we have
    $$
    \E \sup_{x \in T} \abs[\Big]{\norm*{Ax}_2 - \lambda \norm*{x}_2} \le C K \gamma(T).
    $$
    Moreover, for any $u \geq 0$, with probability at least $1 - 3e^{-u^2}$, we have
    $$
    \sup_{x \in T} \abs[\Big]{\norm*{Ax}_2 - \lambda \norm*{x}_2} \le C K \left[w(T) + u\cdot \rad(T)\right].
    $$ 
\end{corollary}

\medskip

\begin{remark}[Why is column normalization needed?]
In contrast to the previous result (Corollary~\ref{thm: old result}), here we require a hard normalization of columns:
\begin{equation}	\label{eq: columns normalized}
	\norm{A_i}_2 = \lambda \quad \text{a.s. for all } i.
\end{equation}
One cannot remove this assumption even for $n=1$, or replace it with a weaker assumption of isotropy as in Theorem~\ref{thm: old result}. This is because a mean-zero isotropic subgaussian random vector $A_1$ might vanish with probability $1/2$.  For example, if $\delta\sim \mbox{Bern}(1/2)$ and $X$ is (independently) chosen uniformly from $\sqrt{m}S^{m-1}$, then $\delta X$ is a subgaussian vector with $O(1)$ subgaussian norm.  If the columns of $A$ are i.i.d.~copies of $\delta X$ and $T = \{e_1\}$, the first canonical basis vector, then
\[
\sup_{x \in T} \abs[\Big]{\norm*{Ax}_2 - \lambda \norm*{x}_2} = |\delta \sqrt{m} - \lambda| \geq \sqrt{m}/2\]
with probability $1/2$ regardless of the value of $\lambda$.  In other words, both the probabilistic bound and bound in expectation become substantially dependent on $m$ without the normalization assumption on the columns.
\end{remark}

\begin{remark}[Column normalization improves the dependence on the subgaussian norm]
    The hard normalization of columns has an interesting beneficial effect on the dependence on the subgaussian norm. In the previous result (Theorem~\ref{thm: old result}) the dependence is $O(K \sqrt{\log K})$, and this dependence is optimal even for matrices with all independent entries \cite{jeong2022sub}. But as soon as we normalize the columns as in Corollary~\ref{cor: lambda cols}, the dependence becomes linear, i.e. $O(K)$. We see in Section \ref{sec: sparse matrices} that this small improvement leads to completely different asymptotic behaviour for the embedding guarantees of sparse matrices.  We also observe in Remark~\ref{rem: normalize rows} that this improvement only occurs for matrices with independent {\em columns} of a fixed norm, not for matrices with independent {\em rows} of a fixed norm.    
\end{remark}

\begin{remark}[Comparison to \cite{bartl2022random}]
\label{rem: comparison}
We now compare to the results in \cite{bartl2022random}.  Therein, the authors require the columns of $A$ to satisfy:  1) a ``thin-shell condition", i.e., to have concentrated norms, and 2) a moment condition on the marginals.  This allows a broad class of distributions, including subgaussian.  We specialize their results to the setting of Corollary~\ref{cor: lambda cols} with $\lambda = \sqrt{m}$, and also assume that a) the columns of $A$ are isotropic, b) $T \subset S^{n-1}$, and c) $\gamma(T) \geq \sqrt{\log n}$ to allow a simpler comparison.  In this case, the main result in \cite{bartl2022random} (Theorem 1.5), yields the following.  For $u\geq 1$, with probability at least $1 - 2\exp(-c \gamma(T)^2) - n^{-u}$, 
\[
\sup_{x \in T} \abs[\Big]{\norm*{Ax}^2_2 - m \norm*{x}^2_2} \le C(K, u) \cdot \sqrt{m} \left[ \gamma(T)(1 + \gamma(T)/\sqrt{m})\right] \log(e n/\gamma^2(T)).\] 
Note that this is a uniform bound on the squared difference $\norm*{Ax}^2_2 - m \norm*{x}^2_2 = (\norm*{Ax}^2 - \sqrt{m} \norm*{x}_2)(\norm*{Ax}_2 + \sqrt{m} \norm*{x}_2)$ thus explaining the extra factor of $\sqrt{m}$.  Further, in typical cases of interest $\gamma(T) < \sqrt{m}$.  Thus, it is natural to compare $C(K,\mu) \gamma(T) \log(en/\gamma^2(T))$ to the error bound given in Corollary \ref{cor: lambda cols}, which takes the form $C K \gamma(T)$.  In other words, by concentrating on the subgaussian case, we are able to improve on the results in \cite{bartl2022random} by giving the
 explicit and optimal {\em linear} dependence on the subgaussian norm $K$, and also removing the $\log(en/\gamma(T)^2)$ factor --- thereby showing that there is no dependence on the ambient dimension.
\end{remark}

We now specify some notations, discuss applications to sparse matrices in Section \ref{sec: sparse matrices}, show that column normalization can improve embedding guarantees for i.i.d.~matrices in Section \ref{sec: column normalization}, prove our main result in Section \ref{sec: proofs}, then our auxiliary results in Section \ref{sec: auxiliary proofs}, and then conclude with an open question in Section \ref{sec: conclusion}.

\subsection{Notation}  We use $\norm*{\cdot}_2$ for the Euclidean norm of vectors, and $S^{n-1}$ for the set of vectors in $\R^n$ with Euclidean norm $1$.  We say that $f \lesssim g$ if there exists an absolute constant $C$ such that $f \leq C g$ everywhere. The symbols $C, c$ generally refer to absolute constants which may vary from instance to instance. We say $f \asymp g$ if $c f \leq g \leq C f$, i.e., $f$ and $g$ are equivalent up to constants. 

\section{An example: sparse matrices} \label{sec: sparse matrices}
Sparse random matrices are a central object of study in the context of sparse Johnson Lindenstrauss embeddings \cite{achlioptas2003database} and sparse subspace embeddings \cite{clarkson2013low}. 
These give a method to facilitate computation by embedding data into a lower dimension near isometrically via a sparse (and thus low complexity) transform.  The central question is: \textit{How should we choose the distribution of $A$ to (a) minimize the number of nonzero entries and (b) ensure it behaves like a near-isometry?}


An interesting phenomenon has been observed (and understood) in this field: Choosing $A$ with i.i.d.~entries limits its ability to near-iso\-metric\-ally embed sparse vectors \cite{matouvsek2008variants, ailon2009fast}. However, choosing the columns from a distribution with exactly $s$ non-zero entries per column can give a significant improvement.  The literature in this area generally considers embeddings of singletons, finite sets, or subspaces, with the notable exception in \cite{bourgain2015toward}, discussed further below.  We extend the understanding of this phenomenon to arbitrary bounded sets in the following.
 We begin with an illustrative example to give an intuition for this point.  

\subsection{Approximate vs. exact sparsity}

Consider two different models to generate a sparse $m \times n$ matrix with (about) $s$ nonzero entries per column. In the first, ``approximate'' model, we sample each entry independently from the $\Ber(s/m)$ distribution. In the second, ``exact'' model, we put $s$ ones uniformly at random in each column of the matrix. Although individual entries in both models have the same $\Ber(s/m)$ distribution, it turns out that these two models have diverging behaviour as $m$ increases. In the first model, the number of nonzero entries in a column converges to a $\mbox{Poisson}(s)$ distribution. Thus, the column norms will concentrate around $\sqrt{s}$, but their deviations will not decrease to zero as $m$ increases. Equivalently, the matrix will randomly distort the norm of a standard basis vector, and {\em that distortion will never decrease to zero}. In contrast, in the second model, the number of nonzero entries per column is exactly $s$ and so the column norms are exactly $\sqrt{s}$. Further, as $m$ increases, the columns become more and more likely to have disjoint support and thus become orthogonal. Thus, the resulting matrix (multiplied by $1/\sqrt{s}$ for normalization) asymptotically becomes {\em an exact isometry} with probability approaching $1$ as $m$ increases. 

The following two results make these observations quantitative and general. 

\begin{corollary}[Approximately $s$-sparse columns] \label{cor: approx s sparse}
    Let $A$ be an $m \times n$ matrix whose entries are independent random variables which take the value $1$ or $-1$ with probability $\frac{s}{2m}$ each, and $0$ with probability $1-\frac{s}{m}$. Then, for any subset $T \subset \R^n$, we have
    $$
    \E \sup_{x \in T} \abs[\Big]{\norm*{Ax}_2 - \sqrt{s} \norm*{x}_2} \le C \gamma(T).
    $$
    Moreover, this bound is optimal in general. For any integers $1 \le s \le m/2$ and $n \ge 1$ there exists a non-empty subset $T \subset \R^n$ such that
    $$
    \E \sup_{x \in T} \abs[\Big]{\norm*{Ax}_2 - \sqrt{s} \norm*{x}_2} \ge c \gamma(T).
    $$
\end{corollary}

Note that the right-hand side does not decrease when $m$ grows. 

\begin{proof}
    Let $A_i$ denote the columns of $A$. It is not hard to check (see Lemma~\ref{lem: sparse subgauss norm}) that 
    $$
    \norm*{A_i}_\psitwo \lesssim \frac{1}{\sqrt{\log(2/p)}}
    \quad \text{where } p = \frac{s}{m}.
    $$
    If we normalize $A_i$ by setting $Z_i = A_i/\sqrt{p}$, then all coordinates of $Z_i$ are independent random variables with mean zero and variance $1$. Hence $Z_i$ are independent and isotropic random vectors with subgaussian norms bounded by $K = 1/\sqrt{p\log(2/p)}$. Thus, applying Theorem~\ref{thm: old result} for the matrix $A/\sqrt{p}$, we obtain
    $$
    \E \sup_{x \in T} \abs[\Big]{\frac{1}{\sqrt{p}} \norm*{Ax}_2 - \sqrt{m} \norm*{x}_2} 
    \lesssim K\sqrt{\log K}\, \gamma(T)
    \lesssim \frac{\gamma(T)}{\sqrt{p}}.
    $$
    Multiply by $\sqrt{p}$ on both sides to obtain the first part of the theorem.

    To prove the second part, choose $T=\{e_1\}$, the first vector of the standard basis in $\R^n$. Then 
    $$
    \E \sup_{x \in T} \abs[\Big]{\norm*{Ax}_2 - \sqrt{s} \norm*{x}_2} 
    = \E \abs*{\sqrt{Z}-\sqrt{s}}
    \quad \text{(where $Z \sim \Binom(m,p)$)}.
    $$
    Since $\gamma(T) \asymp 1$, it suffices to prove that $\E \abs*{\sqrt{Z}-\sqrt{s}} \gtrsim 1$. A quick computation\footnote{To prove the second identity, express $Z-s=\sum_{i=1}^m X_i$ where $X_i$ are independent centered $\Ber(p)$ random variables. Expand the fourth power of the sum and use that $\E X_i=0$ to get $\E(Z-s)^4 \lesssim m \E X_1^4 + m^2(\E X_1^2)^2 \lesssim mp+m^2p^2 \lesssim s^2$.} gives
    $$
    \E(Z-s)^2 = mp(1-p) \ge \frac{s}{2}
    \quad \text{and} \quad
    \E(Z-s)^4 \lesssim s^2  .
    $$
   Then using Paley-Zygmund inequality for $(Z-s)^2$, we get
    \begin{equation}    \label{eq: Z-s large}
        \Prob*{\abs{Z-s} \ge \frac{\sqrt{s}}{2}}
        =\Prob*{(Z-s)^2 \ge \frac{1}{2}\E(Z-s)^2}
        > c. 
    \end{equation}
    Furthermore, since $\E Z = mp=s$, Markov inequality gives
    $$
    \Prob*{Z \le 2s/c} 
    \ge 1-c/2.
    $$
    Combining this with \eqref{eq: Z-s large} by the union bound, we conclude that with probability at least $c/2$ we have both 
    $$
    \abs{Z-s} \gtrsim \sqrt{s}
    \quad \text{and} \quad 
    Z \lesssim s.
    $$
    If this event occurs, we have 
    $$
    \abs*{\sqrt{Z}-\sqrt{s}}
    =\frac{\abs{Z-s}}{\sqrt{Z}+\sqrt{s}} 
    \gtrsim 1. 
    $$
    Thus $\E \abs*{\sqrt{Z}-\sqrt{s}} \gtrsim (c/2)\cdot 1 \gtrsim 1$, as required to complete the proof.

    The reader may be wondering whether a lower bound can also be shown with sets $T$ for which $\gamma(T)$ is not constant.  Indeed, observe that $\sqrt{Z}-\sqrt{s}$ converges to a $N(0,1/4)$ random variable as $m \rightarrow \infty$ and $s\rightarrow\infty$ provided $s/m \rightarrow 0$ (this can be shown via the Berry-Esseen central limit theorem together with a Taylor approximation).  This can be used as a first step to verify our matching upper and lower bounds for the set $T= \{e_1, e_2, ..., e_n\}$ when $m$ and $s$ are appropriately large.
\end{proof}

\begin{corollary}[Exactly $s$-sparse columns] \label{cor: exact s sparse}
    Let $A$ be an $m \times n$ matrix whose columns are independent random vectors drawn uniformly from the set of vectors in $\R^m$ that have all $0, -1, 1$ entries and exactly $s$ nonzero entries. Then, for any subset $T \subset \R^n$, we have
    $$
    \E \sup_{x \in T} \abs[\Big]{\norm*{Ax}_2 - \sqrt{s} \norm*{x}_2} 
    \le \frac{C \gamma(T)}{\sqrt{\log(2m/s)}}.
    $$
    Moreover, this bound is optimal in general. For any integers $1 \le s \le m$ there exists $n=n(m,s)$ and a non-empty subset $T \subset \R^n$ such that
    $$
    \E \sup_{x \in T} \abs[\Big]{\norm*{Ax}_2 - \sqrt{s} \norm*{x}_2} 
    \ge \frac{c \gamma(T)}{\sqrt{\log(2m/s)}}.
    $$
\end{corollary}

Note that the right-hand side decreases when $m$ grows. 

\begin{proof}
    It is not hard to check (see Lemma~\ref{lem: sparse subgauss norm}) that the columns $A_i$ of $A$ satisfy
    $$
    \norm*{A_i}_2 = \sqrt{s} \text{ a.s.}
    \quad \text{and} \quad
    \norm*{A_i}_\psitwo \lesssim \frac{1}{\sqrt{\log(2m/s)}}.
    $$
    Apply Corollary~\ref{cor: lambda cols} to obtain the first part of the theorem.

    To prove the second part, choose 
    $$
    T = \left\{ e_1-e_2,\, e_1-e_3,\, \ldots, e_1-e_n \right\}
    $$
    where $e_1,\ldots,e_n$ is the standard basis of $\R^n$. Then 
    \begin{align*}
    \Prob*{Ax=0 \text{ for some } x \in T}
        &= \Prob*{A_1=A_i \text{ for some } i=2,\ldots,n} \\
        &= 1-\left(1-\Prob*{A_1=A_2}\right)^{n-1}.
    \end{align*}
    \begin{align*}
        \Prob*{A_1=A_2} 
        &= \frac{1}{2^s\binom{m}{s}} \\
        &\ge \Big(\frac{2em}{s}\Big)^{-s}
            \quad \text{(by a standard bound \cite[Exercise~0.0.5]{hdp-book})} \\
        &\ge \frac{1}{n}
            \quad \text{(if we choose) } n = \left\lceil \Big(\frac{2em}{s}\Big)^{3s} \right\rceil. 
    \end{align*}
    Then 
    $$
    \Prob*{Ax=0 \text{ for some } x \in T}
    \ge 1-\left(1-\frac{1}{n}\right)^{n-1}
    \ge c.
    $$
    When $Ax=0$ for some $x \in T$, we have $\abs[\Big]{\norm*{Ax}_2 - \sqrt{s} \norm*{x}_2} = \sqrt{2s}$. It follows that 
    \begin{equation}    \label{eq: deviation large}
        \E \sup_{x \in T} \abs[\Big]{\norm*{Ax}_2 - \sqrt{s} \norm*{x}_2} 
        \ge c \sqrt{2s}.
    \end{equation}
    On the other hand, denoting by $g_i$ independent $N(0,1)$ random variables, we have
    \begin{equation}    \label{eq: gamma small}
        \gamma(T) 
        = \E \max_{i=2,\ldots,n} \abs{g_1-g_i}
        \lesssim \sqrt{\log n}
        \lesssim \sqrt{s\log(2m/s)}.
    \end{equation}
    where the first inequality follows from the standard bound on the maximum of subgaussian random variables \cite[Exercise~2.5.10]{hdp-book}, and the second inequality follows by our choice of $n$. Comparing \eqref{eq: deviation large} and \eqref{eq: gamma small}, we complete the proof.
\end{proof}

 \begin{remark}[Diverging asymptotic behaviour]
 As shown in Corollaries \ref{cor: approx s sparse} and \ref{cor: exact s sparse}, in general as $m\rightarrow \infty$ the isometry constant for sparse matrices with i.i.d.~entries remains lower bounded, whereas the isometry constant for sparse matrices with exactly $s$ non-zeros per column decreases to 0.
 \end{remark}

\begin{remark}[Alternative treatments of isometry constants of sparse matrices]
We point out three different treatments of the isometry constants of sparse matrices depending on the structure of $T$:
\begin{enumerate}[\quad 1.]
    \item There is a rich understanding of the isometric properties of sparse matrices in the literature on sparse Johnson Lindenstrauss embeddings (see e.g., \cite{kane2014sparser}) when $T$ is a singleton (or a finite set). 
    \item When $T$ is a $d$-dimensional subspace, again there is a thorough understanding in the literature on sparse (oblivious) subspace embeddings.  Here it is shown that one only needs $m=O(d)$ to have an isometry constant lower than 1, even if the sparsity is polylogarithmic (see e.g., \cite{chenakkod2024optimal}) -- a result which does not follow from Corollary~\ref{cor: exact s sparse}. 
    \item We are aware of one paper, \cite{bourgain2015toward}, which addresses the isometric properties of sparse matrices for arbitrary subsets of the sphere.  
    The main theorem in \cite{bourgain2015toward} controls the isometry constant by a new quantity, which is akin to a sparsified gaussian complexity.  We don't repeat that theorem here because the conditions are relatively complex, but note that it recovers the results on sparse subspace embeddings (up to log factors) as a special case.  Thus, the machinery in \cite{bourgain2015toward} allows tighter control of the isometry constant than Corollary \ref{cor: exact s sparse} for some sets and parameter regimes.  However, the results in \cite{bourgain2015toward} do not imply that the restricted isometry constant converges to 0 as $m \rightarrow \infty$ and thus do not imply the diverging asymptotic behaviour of ``approximately sparse" and ``exactly sparse" matrices shown in our work and discussed in the remark above.  Further, we note that our work does not require $T$ to be a subset of the sphere, and is oblivious to every property of $T$ aside from its gaussian complexity.
    \end{enumerate}
\end{remark}

 \begin{remark}[Exactly $s$-sparse rows do not give the same benefit] \label{rem: normalize rows}
 The reader may be wondering whether the improvements gained by choosing the columns from a distribution with exactly $s$ non-zeros could alternatively be gained by choosing the rows from such a distribution.   In fact, this is not the case.  For instance, suppose $A$ is a square $m \times m$ matrix whose rows are independently drawn from the set of vectors that have all $0, -1, 1$ entries and exactly $s$ nonzero entries.  Then, the first column (or any other) of $A$ has entries which are independent random variables which take the value $1$ or $-1$ with probability $\frac{s}{2m}$ each, and $0$ with probability $1-\frac{s}{m}$.  The lower bound in Corollary \ref{cor: approx s sparse} then applies, since it is realized by the set $T=\{e_1\}$ which isolates the first column.
 \end{remark}

\section{Normalizing the columns of an i.i.d.~matrix} \label{sec: column normalization}

We now consider random matrices with independent, symmetric, subgaussian entries.  We show that normalizing the columns improves the dependence of the isometry constant on the subgaussian norm.  We assume the entries have variance 1 which implies that the rows are isotropic, thus fitting the models in \cite{LMPV, jeong2022sub}.  This causes the columns to concentrate around $\sqrt{m}$ and so we normalize them onto $\sqrt{m}S^{n-1}$ to more easily compare to the results in \cite{jeong2022sub}.  We need to condition on the event that the columns have lower-bounded Euclidean norm as otherwise the normalization step could blow up the columns by an arbitrary amount.  This is encoded in the event $F$ below.

\begin{corollary}  \label{cor: normalize cols}  Let $A$ be an $m\times n$ matrix whose entries are independent, symmetric, with unit variance, and subgaussian norm bounded by $K$.   Assume $m \geq C K^2 \log(K) \cdot \log(n)$.\footnote{The assumption that $m \geq C K^2 \log(K) \cdot \log(n)$ could be replaced with $m \geq C K^2 \log(K)$ at the expense of
 losing the lower bound on $\Probevent{F}$.}  Let $A_i$ be the $i$th column of $A$.  Let $F$ be the event that all columns of $A$ have Euclidean norm lower bounded by $\sqrt{m}/2$, that is,
\[F:=\left\{\min_{i = 1, 2, \hdots, n} \norm{A_i}_2 \geq \frac{\sqrt{m}}{2}\right\}.\]
Then $\Pr{F} \geq 1 - 2 \exp(-C m/(K^2 \log K))$.  Further, conditional on $F$, consider the column-normalized matrix $\tilde{A}$ with columns
\[\tilde{A}_i := \frac{\sqrt{m}}{\norm{A_i}_2} \cdot A_i, \qquad i = 1, 2, \hdots, n.\]
The matrix $\tilde{A}$ satisfies 
\[
\E\left[\sup_{x\in T} \Big| \|\tilde{A}x\|_2 - \sqrt{m} \| x\|_2 \Big| \Bigg| F\right] \le C K \gamma(T).
\]
Moreover, conditional on $F$, with probability at least $1 - 3 \exp(-u^2)$.   
\begin{equation} \label{eq: normalized cols}
\sup_{x\in T} \Big| \|\tilde{A}x\|_2 - \sqrt{m} \| x\|_2 \Big| \le C K \left[w(T) + u\cdot \rad(T)\right]. 
\end{equation}
\end{corollary}

Note that one may generate the matrices $A, \tilde{A}$ conditionally on $F$ by re-sampling any columns of $A$ whose Euclidean norms are below $\sqrt{m}/2$.  Further, since $\Pr{F} \geq 1 - 2 \exp(-C m/(K^2 \log K))$, the matrix $\tilde{A}$ is well defined and Equation~\eqref{eq: normalized cols} holds unconditionally with probability at least $(1 - 3 \exp(-u^2)) \cdot \Pr{F} \geq 1 - 3 \exp(-u^2) - 2 \exp(-C m/(K^2 \log K))$.

\begin{remark}[Effect of normalizing columns]
Observe that the rows of $A$ are independent and isotropic.  Thus, we are in the setting of \cite[Corollary 1.2]{jeong2022sub} (see Theorem~\ref{thm: old result} in Section \ref{sec: intro}), in which the authors show the unnormalized version of $A$ satisfies Equation \eqref{eq: normalized cols} but with $K \sqrt{\log K}$ dependence on the subgaussian norm, and that dependence is tight \cite[Proposition 4.5]{jeong2022sub}.  Our corollary shows that normalizing columns removes the logarithmic factor, improving the dependence to linear. 
\end{remark}

As an application of this corollary, we show that normalizing the columns of an ``approximately sparse" matrix improves the embedding guarantees to those of an ``exactly sparse" matrix (see Section \ref{sec: sparse matrices}).  This follows from Corollary~\ref{cor: normalize cols} after rescaling the columns to have norm $\sqrt{s}$.

\begin{corollary} \label{cor: sparse normalized columns} 
    Let $m \geq  s \geq C \log(n)$.\footnote{The assumption that $s \geq C \log(n)$ could be replaced with $s \geq 1$ at the expense of losing the lower bound on $\Probevent{F}$.}  Let $A$ be an $m \times n$ matrix whose entries are independent random variables which take the value $1$ or $-1$ with probability $\frac{s}{2m}$ each, and $0$ with probability $1-\frac{s}{m}$.
Let $F$ be the event that all columns of $A$ have Euclidean norm lower bounded by $\sqrt{s}/2$, that is,
\[F:=\left\{\min_{i = 1, 2, \hdots, n} \norm{A_i}_2 \geq \frac{\sqrt{s}}{2}\right\}.\]
 Then $\Pr{F} \geq 1 - 2 \exp(-C s)$.  Further, conditional on $F$, consider the normalized matrix $\tilde{A}$ with columns
 \[\tilde{A}_i := \frac{\sqrt{s}}{\norm{A_i}_2} A_i, \qquad i = 1, 2, \hdots, n.\]
Conditional on $F$, $\tilde{A}$ satisfies 
 \begin{equation} \label{eq: sparse normalized cols}
 \E\left[\sup_{x\in T} \Big| \|\tilde{A}x\|_2 -  \sqrt{s} \| x\|_2 \Big| \Bigg| F\right] \le \frac{C}{\sqrt{\log(1+m/s)}}\gamma(T). 
 \end{equation}
 \end{corollary}

\begin{remark}
By comparing to Corollary~\ref{cor: approx s sparse}, we see that normalizing the columns has the same beneficial effect as choosing the columns to be exactly $s$-sparse:  Without normalizing, the isometry constant may remain lower bounded as $m \rightarrow \infty$, but after normalizing the isometry constant shrinks to 0 as $m \rightarrow \infty$.
\end{remark}

\section{Proof of Theorem~\ref{thm: main}} \label{sec: proofs}

As in our work \cite{LMPV}, whose exposition can be found in \cite[Section~9.1]{hdp-book}, Theorem~\ref{thm: main} follows automatically from Talagrand's majorizing measure theorem and the following result: 

\begin{theorem}[Subgaussian increments]		\label{thm: subgaussian increments}
	Let $A$ be a random matrix satisfying the assumptions of Theorem~\ref{thm: main}. Then the random process
  	$$
  	Z_x := \norm{Ax}_2 - \norm{x}_2
  	$$
  	has subgaussian increments, namely 
  	$$
  	\norm{Z_x - Z_y}_\psitwo 
  	\le C K \norm{x-y}_2 \quad \text{for all } x,y \in \R^n.
  	$$
\end{theorem}

To see how this leads to Theorem~\ref{thm: main}, let us show how to combine Theorem~\ref{thm: subgaussian increments} with the specialized version of Talagrand’s comparison inequality below. This result is from \cite{LMPV}; for more general versions, see \cite[Section 8]{hdp-book}.

\begin{theorem}[Comparison inequality] \label{theorem_majorm} Let $(Z_x)_{x\in T}$ be a random process on a bounded set $T\subset \R^n$. 
Assume that the process has subgaussian increments, that is there exists $M\geq 0$ such that \[ \|Z_x-Z_y\|_{\psi_2} \leq M\|x-y\|_2 \;\text{ for all }\; x,y \in T. \] 
Then \[ \E \sup_{x,y\in T}|Z_x-Z_y|\leq C M \, \E\sup_{x\in T}\ip{g}{x}, \]
 where $g\sim N(0,I_n)$. Moreover, for any $u\geq 1$, the event \[ \sup_{x,y\in T}|Z_x-Z_y|\leq CM \left( \E\sup_{x\in T}\ip{g}{x}+u\cdot \mathrm{diam}(T) \right) \] holds with probability at least $1-e^{-u^2}$. Here $\mathrm{diam}(T):=\sup_{x,y\in T}\|x-y\|_2$. \end{theorem}
The first part of Theorem~\ref{theorem_majorm} can be found in \cite[Theorem 2.4.12]{talagrand2014upper} and the second part can be found in \cite[Theorem 3.2]{2015dirksen}.  See also \cite[Theorem 8.5.5]{hdp-book}.

Theorem~\ref{thm: main} then follows by observing that for any $x_0 \in T$,
\[\sup_{x \in T} \abs{Z_x} \leq \abs{Z_{x_0}} + \sup_{x \in T} \abs{Z_x - Z_{x_0}} \leq \abs{Z_{x_0}} + \sup_{x,y \in T} \abs{Z_x - Z_y}.\]
The second summand is controlled by Theorems \ref{theorem_majorm} and \ref{thm: subgaussian increments}.  The first summand is a single subgaussian random variable.  Indeed, by Proposition~\ref{prop: conc of norm} below, $\norm{Z_{x_0}}_{\psi_2} \lesssim K \norm{x_0}_2 \leq K \rad(T)$.  It is relatively straightforward to complete the proof of Theorem \ref{thm: main}, see \cite[Section 4]{jeong2022sub} for the full detailed argument.

We now build the proof of Theorem \ref{thm: subgaussian increments}, starting with two special cases.

\subsection{Unit vector $x$, zero vector $y$}

Let us first prove a partial case of Theorem~\ref{thm: subgaussian increments} for $\norm{x}_2=1$ and $y=0$. 

\begin{proposition}[Concentration of norm]	\label{prop: conc of norm}
	Let $A$ be a random matrix satisfying the assumptions of Theorem~\ref{thm: main}. Then for any $x \in S^{n-1}$ we have
	$$
	\norm[\big]{ \norm*{Ax}_2 - 1}_\psitwo \lesssim K. 
	$$
\end{proposition}

To prove this result, let us first note that $Ax$ is a subgaussian random vector:

\begin{lemma}[Sum of independent subgaussians]	\label{lem: sum of subgaussians}
	Let $X_1,\ldots,X_n$ be independent, mean zero, subgaussian random vectors in $\R^m$. Then 
	$$
	\left\|\sum_{i=1}^n X_i\right\|_\psitwo^2 \le C \sum_{i=1}^n \norm{X_i}_\psitwo^2.
	$$
	In particular, under the assumptions of Theorem~\ref{thm: main}, for any unit vector $x \in \R^n$ we have
	$$
	\norm{Ax}_\psitwo \le CK. 
	$$
\end{lemma}

\begin{proof}[Proof of Lemma~\ref{lem: sum of subgaussians}.]
	The first part is an easy consequence of Definition \eqref{eq: sub-gauss vector} and \cite[Proposition~2.6.1]{hdp-book}. The second part follows if we expand $Ax = \sum_{i=1}^n x_i A_i$ and apply the first part to $X_i = x_i A_i$. 
\end{proof}

We now proceed with a moment generating function (MGF) bound.  
For any fixed $\l \in \R$, consider the function 
\begin{equation}	\label{eq: Phi lambda}
	\Phi_\lambda(z) \coloneqq \exp(\lambda z). 
\end{equation}
Then the MGF of a random variable $Z$ can be expressed as $\E \Phi_\lambda(Z)$.

\begin{lemma} \label{lem: mgf} Let $A$ and $x$ be as in Proposition \ref{prop: conc of norm}.  Set $S:= \norm{Ax}_2^2 - 1$.  Then
\[
\E\Phi_\lambda(S) 
\leq \exp(C \lambda^2 K^2) 
\quad \mbox{for} \quad 
\abs{\lambda} \leq \frac{c}{K^2}.
\]
\end{lemma}

\begin{proof}
Expanding the norm, we get
	$$
	\norm{Ax}_2^2 
	= \norm{\sum_{i=1}^n x_i A_i}_2^2
	= \sum_{i=1}^n x_i^2 \norm{A_i}_2^2 + \sum_{i \ne j} x_i x_j \ip{A_i}{A_j}. 
	$$
	By assumption we have $\norm{A_i}_2=1$ and $\norm{x}_2=1$.  Thus, 
\[\sum_{i=1}^n x_i^2 \norm{A_i}_2^2 = \sum_{i=1}^n x_i^2 = 1\]
an so 
	\begin{equation}	\label{eq: S}
		S 	= \sum_{i \ne j} x_i x_j \ip{A_i}{A_j}. 
	\end{equation}
	Since $\norm{A_i}_2=1$ by assumption, using Cauchy-Schwarz we can see that $S$ is a bounded random variable, so $\Phi_\lambda(S)$ is finite for all $\l$. Let's now get a quantitative bound.

	Fix any $\l \in \R$ and let $A'$ be an independent copy of $A$. Then
	\begin{align*}
	\E \Phi_\lambda(S)
		&= \E \Phi_\lambda \left( \sum_{i \ne j} x_i x_j \ip{A_i}{A_j} \right) \\
		&\leq \E \Phi_\lambda \left(4 \sum_{i,j=1}^n x_i x_j \ip{A_i}{A'_j} \right)
			\quad \text{(decoupling \cite[Exercise~6.1.4]{hdp-book})}\\
		&= \E \Phi_\lambda \left(4 \ip{Ax}{A'x} \right).\\
	\end{align*}
		Let us now condition on $A$.  Recall, by Lemma \ref{lem: sum of subgaussians}, $\norm{A'x}_\psitwo \leq C K$. 
	Then the random variable $\ip{A x}{A' x}$ is (conditionally) subgaussian, and its conditional subgaussian norm is bounded by $C K\norm{A x}_2$, see \cite[Definition~3.4.1]{hdp-book}. Thus
	\begin{align*}
	\E \Phi_\lambda \left( 4\ip{Ax}{A'x} \right)
		&= \E_A \E_{A'} \exp(4 \l \ip{Ax}{A'x}) \\
		&\le \E_A \exp(C \l^2 K^2 \norm{Ax}_2^2)
			\quad \text{(by Equation~\cite[(2.16)]{hdp-book})}.
	\end{align*}
	
	We have shown that
	\[\E \Phi_\lambda(S) \leq \E \exp(C \l^2 K^2 \norm{Ax}_2^2).\]	
	Recall that $S := \norm{Ax}_2^2 - 1$.  We now replace $\norm{Ax}_2^2$ with $S + 1$ in the equation above to yield a recursive equation for the moment generating function of $S$.  This gives 
\begin{equation}  \label{eq: recursive}  
    \E\Phi_\l(S) \leq \exp(C \l^2 K^2) \cdot \E \exp(C \l^2 K^2 S)
    \quad \text{for all } \l \in \R.
\end{equation}

Now assume that
$$
0 \le \l \leq 1/(2 C K^2).
$$
Then $0\leq C\l K^2 \leq 1/2$, and Jensen inequality gives
\[
\E \exp(C \l^2 K^2 S) 
= \E \big[ \exp(\l S)^{C\l K^2} \big] 
\leq \left[\E\exp(\l S)\right]^{C \l K^2} 
= [\E\Phi_\l(S)]^{C \l K^2}.
\]
Insert this into Equation~\ref{eq: recursive} and rearrange the terms to get
$$
\E\Phi_\l(S) 
\leq \exp(C \l^2 K^2 )^{1/(1 - C\l K^2)}.   
$$
Since $1 - C\l K^2 \geq 1/2$ by assumption, we conclude that 
\begin{equation}    \label{eq: Phi bound for positive}
    \E\Phi_\l(S) \leq \exp(2C \l^2 K^2 )
    \quad \text{for all } 0 \le \l \leq 1/(2 C K^2),
\end{equation}
as desired.

Finally, let us lift the nonnegativity assumption and simply assume that 
$$
\abs{\l} \leq 1/(2 C K^2).
$$
We have 
\begin{align*}
    \E \exp(C \l^2 K^2 S)
    &= \E \Phi_\mu(S)
        \quad \text{(where $\mu = C \l^2 K^2$)} \\
    &\le \exp(2C \mu^2 K^2 )
        \quad \text{(using \eqref{eq: Phi bound for positive}, since $0 \le \mu \leq 1/(2 C K^2)$)} \\
    &\le \exp(2C \l^2 K^2 )
        \quad \text{(since $\mu \le \abs{\l})$}.
\end{align*}
Substitute this into \eqref{eq: recursive} to get 
$$
\E\Phi_\l(S) \leq \exp(3C \l^2 K^2 )
\quad \text{for all } \abs{\l} \leq 1/(2 C K^2).
$$
The proof is complete.
\end{proof}

\begin{remark}[Why normalized columns give linear dependence on $K$]
We can now point out the first place where the assumption of normalized columns allows linear dependence on $K$ in our main result.  In the proof above, 
\[S = \norm{Ax}_2^2 - 1 = \sum_{i=1}^n x_i^2 (\norm{A_i}_2^2-1) + \sum_{i \ne j} x_i x_j \ip{A_i}{A_j}.\]
Since the columns of $A$ have norm 1, the first term disappears, leaving just a linear combination of the inner products of different columns (see Equation~\eqref{eq: S}).  After decoupling, this gives the neat recursive equation \eqref{eq: recursive} for $S$ which ultimately allows a tight control of its tail.  If we had not assumed normalized columns then the variations in the norms of $A_i$ would have necessarily not allowed this (as follows from \cite[Proposition 4.5]{jeong2022sub}).
\end{remark}

The following lemma converts the bound on the moment generating function of $S$ into a tail bound.

\begin{lemma} \label{lem: bernstein}
Let $S$ be as in Lemma \ref{lem: mgf}. Then for any $t \geq 0$
\[\Pr{\abs{S} \geq t} \leq 2\exp\left(-\frac{C}{K^2}\left(t^2 \wedge t \right)\right).\]
\end{lemma}

\begin{proof}
Let $0 \le \l \leq c/K^2$, as in Lemma \ref{lem: mgf}. Then
$$
\Prob*{S \ge t}
= \Prob*{e^{\l S} \ge e^{\l t}}
\le e^{-\l t} \E \Phi_\l(S)
\le \exp(-\l t + C \l^2 K^2),
$$
where the last inequality follows from Lemma \ref{lem: mgf}.

The optimal choice of $\l$ subject to the constraint $\l\leq c/K^2$ is 
$\l = t/(2 C K^2) \wedge c/K^2$. With this choice, we obviously have $\l \leq t/(2 C K^2)$, which is equivalent to    
$C \l^2 K^2 \leq \l t/2$. Thus,  
\[\Pr{S \ge t} \le \exp(-\l t + C\l^2 K^2 )\leq\exp(-\l t/2) \le \exp\left(-\frac{C'}{K^2} \left(t^2 \wedge t\right)\right).\]
$\Pr{-S \ge t}$ may be bounded similarly, and the the proof is completed via a union bound.
\end{proof}

\begin{proof}[Proof of Proposition~\ref{prop: conc of norm}]

By Lemma \ref{lem: bernstein}, for $t \geq 0$
$$
\Prob*{\abs*{\norm{Ax}_2^2-1} \ge t} 
\le 2\exp\left( -\frac{C}{K^2}\left(t^2 \wedge t\right) \right).
$$
Now we use the following elementary observation, which holds for all $z,\d \geq 0$:
$$
\abs{z-1} \ge \d \quad \text{implies} \quad \abs{z^2-1} \ge \d \vee \d^2.
$$
We obtain 
$$
\Prob*{\abs*{\norm{Ax}_2 - 1} \geq \d} 
\le \Prob*{\abs*{\norm{Ax}_2^2 - 1} \geq \d \vee \d^2}
\le 2\exp\left( -\frac{C\d^2}{K^2} \right),
$$
since if $t=\d \vee \d^2$ then $t^2 \wedge t = \d^2$. The proof is complete.
\end{proof}

\subsection{Vectors on the sphere, squared process}

We now show that $\norm{Ax}_2^2 - \norm{Ay}_2^2$ has a mixed tail bound for $x,y \in S^{n-1}$.  We will then convert that into a subgaussian tail bound for $Z_x - Z_y$ in the next section.

\begin{lemma}[Squared process] \label{lem: squared process}
	Let $A$ be a random matrix satisfying the assumptions of Theorem~\ref{thm: main}. 
	Fix any pair of vectors $x,y \in S^{n-1}$.  
	Then the random variable 
	$$
	\Delta \coloneqq \frac{\norm{Ax}_2^2 - \norm{Ay}_2^2}{\norm{x-y}_2}
	$$
	satisfies 
	$$
	\Pr{\abs{\Delta} \ge t} 
	\le 2\exp\left(-\frac{C}{K^2}\left(t^2 \wedge t\right) \right).
	$$
\end{lemma}

\begin{proof}
Set
\[u := x+y, \quad v := x - y.\]

Then

\[\norm{Ax}_2^2 -\norm{Ay}_2^2 = \ip{Au}{Av} = \sum_{i} \ip{A_i}{A_i} u_i v_i +  \sum_{i\neq j} \ip{A_i}{A_j} u_i v_j =: I + II.\]
Observe that $u_i v_i = x_i^2 - y_i^2$ and $\ip{A_i}{A_i} = \norm{A_i}_2^2 = 1$.  Thus
\[I = \sum_i x_i^2 - y_i^2 = \norm{x}_2^2 - \norm{y}_2^2 = 1 - 1 = 0.\]

Now, set $\hat{v} = v/\norm{v}_2$ so

\[\Delta = \frac{\norm{Ax}_2^2 - \norm{Ay}_2^2}{\norm{v}_2} =  \sum_{i\neq j} \ip{A_i}{A_j}  u_i \hat{v}_i.\]

Using decoupling (\cite[Exercise~6.1.4]{hdp-book}) as in the proof of Lemma \ref{lem: mgf}, we let $A'$ be an independent copy of $A$ and get
$$
\E\Phi_\l(\Delta)
\le \E \Phi_\l  \left[4 \ip{Au}{A'\hat{v}}\right].
$$

As in the proof of Lemma~\ref{lem: mgf}, taking the conditional expectation with respect to $A'$ gives
\[\E\Phi_\l(\Delta) \leq \E\exp(C \l^2 K^2 \norm{A u}_2^2).\]
Observe that $\norm{u}_2 =\norm{x + y}_2 \leq \norm{x}_2 + \norm{y}_2 = 2$ and assume $\abs{\l} \leq 1/(C K^2)$.  Then $C \l^2 K^2 \leq \abs{\l}$ and by Lemma~\ref{lem: mgf}
\[\E\exp(C \l^2 K^2 \norm{A u}^2) \leq \exp(C \l^2 K^2 ).\]

Thus, 
\[\E\Phi_\l(\Delta) \leq \exp(C \l^2 K^2) \qquad \mbox{for} \quad \abs{\l} \leq \frac{c}{K^2}.\]

To finish, observe that this is exactly the same bound on the MGF that we had previously for $S = \norm{A x}_2^2 -1$ (as in Lemma \ref{lem: mgf}).  Thus, it can be converted into the same tail bound as in Lemma~\ref{lem: bernstein}, but with $S$ replaced by $\Delta$, as desired.
\end{proof}

\subsection{Vectors on the sphere, original process}

We now show that the random process $Z_x = \abs{\norm{Ax}_2 - 1}$ has subgaussian increments for vectors on the sphere, i.e., we show that
\[\norm{Z_x - Z_y}_{\psi_2} \leq C K \norm{x-y}_2\]
when $\norm{x}_2 = \norm{y}_2 =1$.  

This follows from the following lemma, together with the observation that
\[\abs{Z_x - Z_y} = \abs{\abs{\norm{Ax}_2 - 1} - \abs{\norm{Ay}_2 - 1}} \leq \abs{\norm{Ax}_2 -\norm{Ay}_2}.\]

\begin{lemma} \label{lem: increments on sphere} Fix any pair of vectors $x,y \in S^{n-1}$.  Then
\[\norm*{\norm{Ax}_2 - \norm{Ay}_2}_{\psi_2} \leq C K \norm{x - y}_2.\]
\end{lemma}
\begin{proof}
This proof is essentially the same as the proof of \cite[Lemma 5.4]{LMPV}.  We repeat it here for convenience.

Assume $x \neq y$ (as otherwise the result clearly holds) and set
\[Q:= \frac{\abs{\norm{Ax}_2 - \norm{Ay}_2}}{\norm{x - y}_2}.\]

We wish to show that 
\begin{equation}\label{eq: subgauss tail}
\Pr{Q \geq t} \leq C \exp\left(-C \frac{t^2}{K^2}\right) \quad \mbox{for } t \geq 0. 
\end{equation}

We show this separately for small $t$ and large $t$.  Set $v = x - y$ and $\hat{v} = v/\norm{v}_2$.  

\textit{Case 1:}  $t \geq 2$.  

By the triangle inequality, 
$\abs{\norm{Ax}_2 - \norm{Ay}_2} \leq \norm{A(x - y)}_2 = \norm{Av}_2$.  Thus, 
\[Q \leq \norm{A\hat{v}}_2.\]

Then,
\[\Pr{Q \geq t} \leq \Pr{\norm{A \hat{v}}_2 \geq t} = \Pr{\norm{A \hat{v}}_2-1 \geq t-1} \leq \Pr{\norm{A \hat{v}}_2-1 \geq t/2}\]
since $t \geq 2$.  Recall that  $\norm{\norm{A \hat{v}}_2-1}_{\psi_2} \lesssim K$ by Proposition \ref{prop: conc of norm} and so 
\[\Pr{\norm{A \hat{v}}_2 -1 \geq t/2} \leq 2\exp\left(-C\frac{t^2}{K^2}\right)\]
as desired.

\textit{Case 2:} $t < 2 $.
Observe that

\[\frac{\abs{\norm{Ax}_2 - \norm{Ay}_2}}{\norm{x - y}_2} =  \frac{\abs{\Delta}}{\norm{Ax}_2 + \norm{Ay}_2} \leq \frac{\abs{\Delta}}{\norm{Ax}_2} \]
with $\Delta$ defined as in Lemma~\ref{lem: squared process}.  

Define the events $F = \{\abs{\Delta} \geq t/2\}$ and $G=\{\norm{Ax}_2 \leq 1/2\}$ so that 
\[\left\{\frac{\abs{\Delta}}{\norm{Ax}_2} \geq t\right\} \subset F\cup G.\]
By the union bound 
\[\Pr{\frac{\abs{\Delta}}{\norm{Ax}_2} \geq t} \leq \Probevent{F} + \Probevent{G}.\]
By Lemma~\ref{lem: squared process}, $\Probevent{F} \leq 2\exp(-C t^2/K^2)$ and by Proposition~\ref{prop: conc of norm},
\[\Probevent{G} = \Pr{\norm{Ax}_2 - 1 \leq -\frac{1}{2}} \leq 2\exp\left(-\frac{C}{K^2}\right) \leq 2 \exp\left(-\frac{C t^2}{K^2}\right)\]
since $t/2 \leq 1$ (we have absorbed two factors of 4 into the numeric constant $C$).

Thus, 
\[\Probevent{F} + \Probevent{G} \leq 4 \exp\left(-C \frac{t^2}{K^2}\right)\]
giving the desired bound on $\Pr{Q \geq t}$.
\end{proof}

\subsection{Arbitrary vectors, original process}
We now show that 
\begin{equation} \label{eq: increments}
\norm{Z_x - Z_y}_{\psi_2} \lesssim K \norm{x - y}_2
\end{equation}
 for any $x,y\in \R^n$, thus completing the proof of Theorem \ref{thm: subgaussian increments}.  Again, we present this for convenience, since the argument is already made in \cite{LMPV} (see also \cite[Section~9.1.4]{hdp-book}). 
 
Without loss of generality, assume $\norm{x}_2\geq \norm{y}_2$.  Assume $y \neq 0$ as otherwise Equation \eqref{eq: increments} is already shown by Proposition \ref{prop: conc of norm}.  By a simple rescaling argument, we may assume $\norm{y}_2=1$.  Let $\hat{x} = x/\norm{x}$, i.e., the projection onto the sphere (or the convex hull of the sphere).  Since projection onto a convex set is contractive, and by definition of projection, we have
\begin{equation}\label{eq: contractive}
\norm{\hat{x} - y}_2 \leq \norm{x - y}_2 \qquad \mbox{and} \qquad \norm{x - \hat{x}}_2 \leq \norm{x - y}_2.
\end{equation}

By the triangle inequality, $\norm{Z_x - Z_y}_{\psi_2} \leq \norm{Z_x - Z_{\hat{x}}}_{\psi_2} + \norm{Z_{\hat{x}} - Z_y}_{\psi_2}.$  We now use the results from the previous sections to show that each term is bounded by $C K\norm{x - y}_2$, thus completing the proof.  Indeed, by linearity, $Z_x = \norm{x}_2 \cdot Z_{\hat{x}}$.  Further, $\norm{x}_2 - 1 = \norm{x - \hat{x}}_2$ and so  
\[\norm{Z_x - Z_{\hat{x}}}_{\psi_2}  = \norm{x - \hat{x}}_2 \cdot \norm{\norm{A\hat{x}}_2 -1}_{\psi_2} \leq CK \norm{x - \hat{x}}_2 \leq CK \norm{x - y}_2.\]
The first inequality follows by Proposition~\ref{prop: conc of norm} and the second by Equation~\eqref{eq: contractive}.

Next, by Lemma~\ref{lem: increments on sphere} followed by Equation~\eqref{eq: contractive}, 
\[\norm{Z_{\hat{x}} - Z_y}_{\psi_2} \leq C K \norm{\hat{x} - y}_2 \leq CK \norm{x - y}_2.\]

The proof is finished.
\section{Proofs of the material in Sections~\ref{sec: sparse matrices},~\ref{sec: column normalization}} \label{sec: auxiliary proofs}
We begin with a lemma controlling the subgaussian norm of the columns of ``approximately sparse" and ``exactly sparse" matrices, as referenced in Sections~\ref{sec: sparse matrices},~\ref{sec: column normalization}.

\begin{lemma} \label{lem: sparse subgauss norm}
Let $s \leq m $ be a positive integer and consider random vectors $X, Y \in \R^m$.  The vector $X$ is drawn uniformly from the set of vectors in $\R^m$ with that have all $0, 1, -1$ entries and exactly $s$ non-zero entries. The entries of $Y$ are independent and take the value $1$ or $-1$ with probability $\frac{s}{2m}$ each, and $0$ with probability $1 - \frac{s}{m}$.  Then,  
\begin{equation}\label{eq: s-sparse vector}
\norm{X}_{\psi_2} \leq \frac{C}{\sqrt{\log(1+ m/s)}}  \quad \mbox{and} \quad  \norm{Y}_{\psi_2} \leq \frac{C}{\sqrt{\log(1+ m/s)}}.
\end{equation}
\end{lemma}
\begin{proof}
It is not hard to show that by definition, $\norm{Y_i}_{\psi_2} = 1/\sqrt{\log(1 + m/s)}$.  The bound on $\norm{Y}_{\psi_2}$ then follows from \cite[2.6.1]{hdp-book} and Definition~\eqref{eq: sub-gauss vector}.  

We now bound the moments of the marginals of $X$ by the moments of the marginals of $Y$, thus controlling $\norm{X}_{\psi_2}$ by $\norm{Y}_{\psi_2}$. Below, given an event $F$ and a random variable $Z$, we use the notation
\[\norm{Z|F}_{\psi_2} := \inf\{t > 0: \E[\exp(Z^2/t^2)|F]\}\]
i.e., it is the subgaussian norm of $Z$ conditioned on $F$.

Let $u \in S^{n-1}$ and let $G$ be the event $G := \{\norm{Y}_0 \geq s\}$ where we use $\norm*{\cdot}_0$ for the number of nonzero entries of a vector.  Since $s$ is the median of $\norm{Y}_0$, we have $\Pr{G} \geq 1/2$ (see \cite{kaas1980mean}).  Let $p \geq 1$ and observe that
\[\E\abs{\ip{Y}{u}}^{p} \geq  \E\left[\abs{\ip{Y}{u}}^{p}\Big|G\right]\cdot \Pr{G} \geq \E\left[\abs{\ip{Y}{u}}^p \Big| \norm{Y}_0 = s\right] \cdot \frac{1}{2} = \E\abs{\ip{X}{u}}^{p}\cdot \frac{1}{2}.\]
The last equality follows since conditioned on $\norm{Y}_0 = s$, $Y$ has the same distribution as $X$.  We have shown that
\[\E\abs{\ip{X}{u}}^{p} \leq 2 \E\abs{\ip{Y}{u}}^{p}.\]
It follows from the moment-based characterization of the subgaussian norm \cite[Equation 2.15]{hdp-book} that $\norm{X}_{\psi_2} \leq C \norm{Y}_{\psi_2}$ as desired.
\end{proof}




\subsection{Proof of Corollary~\ref{cor: normalize cols}}
Our strategy is to condition on the event $F$ that the column norms of $A$ are lower bounded, show that conditionally on this event $\tilde{A}$ has independent, mean-zero, subgaussian columns, and then apply Corollary \ref{cor: lambda cols}.

The norm of a vector with independent subgaussian entries is well concentrated.  Indeed, by \cite[Theorem 4.1]{jeong2022sub}
the columns of $A$ satisfy,
\[\norm{\norm{A_i}_2 - \sqrt{m}}_{\psi_2} \lesssim K\sqrt{\log K} \qquad i = 1, 2, \hdots, n\]
and thus, 
\[\Pr{\norm{A_i}_2 \leq \frac{\sqrt{m}}{2}} = \Pr{\norm{A_i}_2 - \sqrt{m} \leq -\frac{\sqrt{m}}{2}} \leq 2 \exp\left(-C \frac{m}{K^2 \log K}\right).\]
Then, by the union bound, the event
\[F:=\left\{\min_{i = 1, 2, \hdots, n} \norm{A_i}_2 \geq \frac{\sqrt{m}}{2}\right\}\]
satisfies 
\[\Pr{F} \geq 1 - 2n\exp\left(-C\frac{m}{K^2 \log K}\right)=1 - 2\exp\left(\log(n) -C\frac{m}{K^2 \log K}\right).\]
By assumption, $m \geq CK^2 \log(K)\cdot \log(n)$, and so the $\log(n)$ term is insignificant, only changing the constant:
\[\Pr{F} \geq  1 -2 \exp \left(-C\frac{m}{K^2 \log K}\right) \geq \frac{1}{2}.\]

Now condition on the event $F$.  Clearly, the columns of $A$ are (conditionally) non-zero.  Further,

\begin{equation} \label{eq: conditional sub-gauss norm}
\norm{\tilde{A_i}|F}_{\psi_2} = \left\|\frac{\sqrt{m}}{\norm{A_i}_2} A_i \bigg |F\right\|_{\psi_2} \leq 2 \norm{A_i|F}_{\psi_2} \leq C \norm{A_i}_{\psi_2}.
\end{equation}

The first inequality follows because $\sqrt{m}/\norm{A_i} \leq 2$ (conditionally).  The second inequality follows because 
conditioning on an event with probability at least $0.5$ cannot increase the subgaussian norm of $A_i$ by more than a constant factor.  For instance, this follows by checking the moment condition in \cite[Proposition 2.5.2]{hdp-book}.

Thus, by \cite[Proposition 2.6.1]{hdp-book},
\[\norm{\tilde{A_i}|F}_{\psi_2} \leq C\norm{A_i}_{\psi_2} \leq C K.\]
Further, since the entries of $A_i$ are symmetric, and thus the signs of those entries are independent from their magnitudes, the conditional entries of $\tilde{A_i}$ are still symmetric.  It follows that $\E[\tilde{A_i} |F] = 0$.  Summarizing, conditional on $F$, $\tilde{A}$ has the following properties
\[\norm{\tilde{A_i}}_2 = \sqrt{m}, \quad \norm{\tilde{A_i}|F}_{\psi_2} \leq CK , \quad \E[\tilde{A_i}|F] = 0.\]

Clearly, the columns of $\tilde{A}$ remain independent conditional on $F$ and so we are in the setting of Corollary \ref{cor: lambda cols} with $\lambda = \sqrt{m}$.  The proposition follows.

\begin{remark}[Replacing $m \geq C K^2 \log(K) \cdot \log(n)$ with $m \geq C K^2 \log(K)$] In Equation~\eqref{eq: conditional sub-gauss norm}, conditioning on $F$ is equivalent to conditioning on the event that $\norm{A_i}_2 \geq m/\sqrt{2}$ (by independence).  Thus, we only need $\Pr{\norm{A_i}_2 \geq m/\sqrt{2}} \geq 1/2$ for Equation~\eqref{eq: conditional sub-gauss norm} to hold.  This only requires $m \geq C K^2 \log(K)$, rather than $m \geq C K^2 \log(K) \cdot \log(n)$ (since we may remove the union bound involved in controlling $\Probevent{F}$). Thus Corollary~\ref{cor: normalize cols} holds with the milder assumption on $m$, but without a lower bound on $\Probevent{F}$.  Similar arguments allow us to take the mild assumption $s \geq 1$ in Corollary~\ref{cor: sparse normalized columns} at the expense of losing the lower bound on $\Probevent{F}$.
\end{remark}

\section{Conclusion} \label{sec: conclusion}
We have shown that random matrices with normalized independent columns have similar behaviour that of random matrices with independent rows, except with better dependence on the subguassian norm.  We conclude with an open question: Can the assumption of normalized columns be removed?  This question is made concrete with the following conjecture.  

\begin{conjecture}
  Let $A$ be an $m \times n$ matrix whose columns $A_i$ are independent, mean zero, subgaussian random vectors in $\R^m$ with $\norm*{A_i}_{\psi_2} \leq K$. Then, for any subset $T \subset \R^n$, we have
    $$
    \E \sup_{x \in T} \abs[\Big]{\norm*{Ax}_2 - \norm*{Dx}_2} \le C(K) \gamma(T)
    $$
    where $D$ is the diagonal matrix with $D_{i,i} = \norm{A_i}_2$ and $C(K)$ depends only on $K$.
\end{conjecture}

\bibliography{random-matrices}
\bibliographystyle{abbrv}


\end{document}